\documentclass[a4paper]{article}

\pdfoutput=1 

\usepackage[T1]{fontenc}
\usepackage[utf8]{inputenc}
\usepackage[english]{babel}
\usepackage{graphics}
\usepackage{epsfig}
\usepackage{amsmath}
\usepackage{amssymb}
\usepackage{amsthm}
\usepackage{url}
\usepackage{hyperref}
\usepackage{float}
\usepackage{pifont}
\usepackage[table]{xcolor}
\usepackage{lmodern}

\usepackage{draftwatermark} 
\SetWatermarkText{Preprint} 
\SetWatermarkScale{1}       

\DeclareMathOperator{\tr}{tr}
\DeclareMathOperator{\R}{\mathbb{R}}

\DeclareMathOperator{\e}{e}
\DeclareMathOperator{\D}{\mathrm{d}\hskip-0.4ex}
\newcommand{\T}{\intercal}

\usepackage[normalem]{ulem} 
\newcommand{\hl}[1]{\bgroup\markoverwith
  {\textcolor{#1}{\rule[-.5ex]{2pt}{2.5ex}}}\ULon}

\title{Cross-Gramian-Based Model Reduction: \\ A Comparison}
\author{Christian Himpe\thanks{Contact: \href{mailto:christian.himpe@uni-muenster.de}{\nolinkurl{christian.himpe@uni-muenster.de}}, 
                                        \href{mailto:mario.ohlberger@uni-muenster.de}{\nolinkurl{mario.ohlberger@uni-muenster.de}}, 
Institute for Computational and Applied Mathematics at the University of M\"unster, Einsteinstrasse~62, D-48149 M\"unster, Germany} \and Mario Ohlberger\footnotemark[1]}

\date{}

\newtheoremstyle{thm}{\topsep}{\topsep}{\normalfont \itshape}{}{\normalfont \bfseries}{}{\newline}{}
\theoremstyle{thm}

\newcounter{dummy}
\newtheorem{mytheorem}{Theorem}
\newtheorem{mylemma}{Lemma}
\newtheorem{mycorollary}{Corollary}
\newtheorem{mydefine}[dummy]{Definition}

\begin{document}

\vfill

\maketitle

\begin{abstract}\bfseries
As an alternative to the popular balanced truncation method, the cross Gramian matrix induces a class of balancing model reduction techniques.
Besides the classical computation of the cross Gramian by a Sylvester matrix equation,
an empirical cross Gramian can be computed based on simulated trajectories.
This work assesses the cross Gramian and its empirical Gramian variant for state-space reduction on a procedural benchmark based to the cross Gramian itself.
\end{abstract}
~

\section{Introduction}
The cross Gramian matrix is an interesting mathematical object with manifold applications in control theory, system theory and even information theory \cite{himpe15d}.
Yet, first and foremost the cross Gramian is used in the context of model order reduction.

The cross Gramian was introduced in \cite{fernando83a} for SISO (Single-Input-Single-Output) systems and extended in \cite{laub83,fernando85} to MIMO (Multiple-Input-Multiple-Output) systems as an alternative balancing method to the balanced truncation \cite{moore81} model reduction technique.
A data-driven variant of the cross Gramian, the empirical cross Gramian, was proposed in \cite{streif06} for SISO systems and extended in \cite{himpe14a} to MIMO systems, expanding the set of empirical Gramians \cite{lall99,lall02}.

Various approaches for cross-Gramian-based model reduction have been studied \cite{aldhaheri91,sorensen01,sorensen02,rahrovani14,himpe14a},
of which this work compares a small selection using a procedural benchmark based on a method to generate random systems introduced in \cite{smith03}.
In this setting, a linear time-invariant input-output system is the central object of interest:
\begin{align}\label{eq:linsys}
\begin{split}
 \dot{x}(t) &= Ax(t) + Bu(t), \\
       y(t) &= Cx(t) + Du(t),
\end{split}
\end{align}
which consists of a dynamical system and an output equation.
The associated vector field is given by a linear transformation of the state $x:\R \to \R^N$ by the system matrix $A \in \R^{N \times N}$,
and a source term introducing the input $u:\R \to \R^M$ through the input matrix $B \in \R^{N \times M}$.
The output $y:\R \to \R^Q$ is determined by an output functional consisting of a linear transformation of the state $x$ by the output matrix $C \in \R^{Q \times N}$,
and a term forwarding the input $u$ by the feed-through matrix $D \in \R^{Q \times M}$;
the latter is assumed to be trivial $D = 0$ in this contribution, as it does not affect the investigated model reduction procedures.

This work is structured as follows:
An outline of the cross Gramian is given in \href{sec:cg}{Section~\ref*{sec:cg}},
which is followed by a summary of the empirical cross Gramian in \href{sec:ecg}{Section~\ref*{sec:ecg}}.
The considered methods for cross-Gramian-based model reduction are presented in \href{sec:mor}{Section~\ref*{sec:mor}}.
In \href{sec:isp}{Section~\ref*{sec:isp}} the procedural benchmark is proposed,
and in \href{sec:num}{Section~\ref*{sec:num}} the considered methods are tested upon this benchmark.

\section{The Cross Gramian}\label{sec:cg}
Two operators play central role in the theory of systems:
The controllability operator $\mathcal{C}:L_2^M \to \R^N$ and the observability operator $\mathcal{O} : \R^N \to L_2^Q$:
\begin{align*}
 \mathcal{C}(u) = \int_{-\infty}^0 \e^{At}Bu(t) \D t, \qquad \mathcal{O}(x_0) = C\e^{At} x_0;
\end{align*}
the former measures how much energy introduced by $u$ is needed to drive $x$ to a certain state,
the latter quantifies how well a change in the state $x$ is visible in the output $y$.
A composition of the observability with the controllability operator yields the Hankel operator\footnote{Commonly described by ``mapping past inputs to future outputs'' \cite{glover84} due to its interpretation as composition of the solution operator with a time-flipping operator \cite{gray98}.} $H:L_2^M \to L_2^Q$,
\begin{align*}
 H = \mathcal{O} \circ \mathcal{C},
\end{align*}
of which the singular values, the so called Hankel singular values, classify the states by importance in terms of the system's input-output coherence.  

The permuted composition of $\mathcal{C}$ with $\mathcal{O}$, that is only admissible for square systems\footnote{A square system has the same number of inputs and outputs $M=Q$.}, yields a cross operator $W_X:\R^N \to \R^N$. 
\begin{mydefine}\label{def:wx}
The composition of the controllability operator $\mathcal{C}$ with the observability operator $\mathcal{O}$ is called \textbf{cross Gramian}\footnote{Since the cross Gramian is generally neither symmetric nor positive semi-definite it is not a Gramian matrix but was introduced under this name in \cite{fernando83a}.} $W_X$:
\begin{align*}
 W_X := \mathcal{C} \circ \mathcal{O} = \int_0^\infty \e^{At} BC \e^{At} \D t \in \R^{N \times N}.
\end{align*}
\end{mydefine}
This cross Gramian concurrently encodes controllability and observability information of the underlying system.

An obvious connection between the Hankel operator and the cross Gramian is given by the equality of their traces\footnote{Similarly, the logarithm-determinants are equal: $\operatorname{logdet}(H) = \operatorname{logdet}(W_X)$, which is the basis for the cross-Gramian-based information index \cite{fu09} measuring information entropy.}:
\begin{align*}
 \tr(H) = \tr(\mathcal{OC}) = \tr(\mathcal{CO}) = \tr(W_X).
\end{align*}

Yet, a central property of the cross Gramian is only available for symmetric systems\footnote{A symmetric system has a symmetric Hankel operator $H =H^*$.}.
\begin{mylemma}
For a symmetric system the absolute values of the eigenvalues of the cross Gramian are equal to the Hankel singular values:
\begin{align*}
 \sigma_i(H) = |\lambda_i(W_X)|.
\end{align*}
\end{mylemma}

This property is expanded to orthogonally symmetric systems in \cite{deabreu86}.

\begin{proof}
 A symmetric system has a symmetric Hankel operator:
 \begin{align*}
  H = H^* \Rightarrow \mathcal{OC} = (\mathcal{OC})^*.
 \end{align*}
 Hence, for the singular values of the Hankel operator holds:
 \begin{align*}
  \sigma_i(H) &= \sigma_i(\mathcal{OC})
              = \sqrt{\lambda_i(\mathcal{OC}(\mathcal{OC})^*)}
              = \sqrt{\lambda_i(\mathcal{OCOC})} \\
              &\stackrel{\text{\cite{hladnik88}}}{=} \sqrt{\lambda_i(\mathcal{COCO})}
              = \sqrt{\lambda_i(W_X W_X)}
              = |\lambda_i(W_X)|.
 \end{align*}
\end{proof}

Classically, to compute the cross Gramian, a relation to the solution of a matrix equation is exploited.
\begin{mylemma}\label{lm:syl}
The cross Gramian is the solution to the Sylvester matrix equation:
\begin{align}\label{eq:syl}
 A W_X + W_X A = -BC.
\end{align}
\end{mylemma}
\begin{proof}
This is a special case of \cite[Theorem~5]{lancaster70}
\end{proof}

\section{Empirical Cross Gramian}\label{sec:ecg}
An alternative approach to the computation of the cross Gramian via a matrix equation is the computation of its empirical variant.
Empirical Gramians \cite{lall99,lall02} result from (numerically obtained) trajectory data.
A justification for this approach is given by the definition of the cross Gramian,
\begin{align*}
 W_X = \int_0^\infty (\e^{At} B) (\e^{A^\T t} C^\T)^\T \D t,
\end{align*}
which can be interpreted as cross covariance matrix of the system's impulse response and adjoint system's impulse response.
As originally in \cite{moore81}, these impulse responses are trajectories,
\begin{align}
 \dot{x}(t) &= Ax(t) + B\delta(t) \Rightarrow x(t) = \e^{At} B, \notag \\
 \dot{z}(t) &= A^\T z(t) + C^\T \delta(t) \Rightarrow z(t) = \e^{A^\T t} C^\T, \notag \\
 \Rightarrow W_X &= \int_0^\infty x(t) z(t)^\T \D t, \label{eq:lwx}
\end{align} 
and yield an \textbf{empirical linear cross Gramian} \cite[Section~2.3]{baur16}.

A more general definition of the empirical cross Gramian \cite{streif06,himpe14a},
without relying on the linear structure of the underlying system\footnote{Such as a closed form for the adjoint system.} is then even applicable to nonlinear systems.
\begin{mydefine}
 For sets $\{c_k \in \R \setminus 0 : l = 1 \dots K\}$,  $\{d_k \in \R \setminus 0 : l = 1 \dots L\}$,
 the $m$-th $M$-dimensional standard base vector $e_{M,m}$ and the $j$-th $N$-dimensional standard base vector $e_{N,j}$, the \textbf{empirical cross Gramian} $\widehat{W}_X \in \R^{N \times N}$ is given by:
 \begin{align}\label{eq:nwx}
 \begin{split}
             \widehat{W}_X &:= \frac{1}{K L M} \sum_{k=1}^K \sum_{l=1}^L \sum_{m=1}^M \frac{1}{c_k d_l} \int_0^\infty \Psi^{klm}(t) \D t, \\
      \Psi_{ij}^{klm}(t) &= \langle x_i^{km}(t) - \bar{x}_i^{km}, y_m^{lj}(t) - \bar{y}_m^{lj} \rangle,
 \end{split}
 \end{align}
with $x_i^{km}$ being the $i$-th component of the state trajectory for the input $u^{km}(t) = c_k e_{M,m} \delta(t)$, zero initial state and $\bar{x}_i^{km}$ the associated temporal average state, while $y_m^{lj}$ is the $m$-th component of the output trajectory for the initial state $x_0^{lj} = d_l e_{N,j}$, zero input and $\bar{y}^{lj}$ the associated temporal average output.
\end{mydefine}
This empirical cross Gramian requires $K \cdot M$ state trajectories for perturbed impulse input with zero steady state, and
$L \cdot N$ output trajectories for perturbed initial states with zero input.
The sets $\{c_k\}$ and $\{d_l\}$ define the operating region of the underlying system and determine for which inputs and initial states the empircal cross Gramian is valid.
In \cite{himpe14a} the empirical cross Gramian is generalized to an empirical cross covariance matrix by admitting arbitrary input functions and centering the state and output trajectories about their respective steady state.
Furthermore, it is shown in \cite{himpe14a} that the empirical cross Gramian is equal to the cross Gramian in \href{def:wx}{Definition~\ref*{def:wx}} for linear systems \eqref{eq:linsys}.

\begin{mytheorem}
 For an asymptotically stable linear system, the empirical cross Gramian $\widehat{W}_X$ reduces to the cross Gramian $W_X$.
\end{mytheorem}
\begin{proof}
See \cite[Lemma~3]{himpe14a}
\end{proof}

Since this empirical cross Gramian requires merely discrete (output) trajectories and does not rely on the linear $\Sigma(A,B,C)$ structure of the system, it can be computed also for nonlinear systems.
Due to the only prerequisite of trajectory data, empirical Gramians are a flexible tool, but warrant prior knowledge on the operating region of the system to define the perturbations.
Hence, based on the idea of numerical linearization cf. \cite{moore79}, empirical Gramians give rise to a data-driven nonlinear model reduction  technique.

The empirical cross Gramian consists of inner products between state trajectories with perturbed input and output trajectories with perturbed initial state.
This allows, by treating the parameters as additional (constant) states, to extend the cross Gramian beyond state input-output coherence to include observability-based parameter identifiability information \cite{himpe14a}.
The associated empirical joint Gramian is an empirical cross Gramian that enables a combined state and parameter reduction from a single cross operator.

Furthermore, a cross Gramian for non-symmetric and also non-square systems \cite{himpe15a},
which can be efficiently computed in its empirical variant, expands the applicability of the cross Gramian to more general system configurations.

\section{Cross-Gramian-Based Model Reduction}\label{sec:mor}
Cross-Gramian-based model reduction is a projection-based approach. 
The state-space trajectory is approximated by a lower-dimensional trajectory,
which results from a reducing truncated projection $R \in \R^{n \times N}$ and a reconstructing truncated projection $S \in \R^{N \times n}$ for $n<N$:
\begin{align*}
 x_r(t) := R x(t) \Rightarrow x(t) \approx S x_r(t).
\end{align*}
Using such projections, a reduced order model for the full order system is given by:
\begin{align*}
 \dot{x}_r(t) &= R A S x_r(t) + R B u_r(t), \\ 
       y_r(t) &= C S x_r(t),
\end{align*}
and $x_{r,0} = R x_0$.
This can be simplified by $A_r := R A S$, $B_r := R B$, $C_r := C R$, due to the linear structure of the system:
\begin{align*}
 \dot{x}_r(t) &= A_r x_r(t) + B_r u(t), \\
       y_r(t) &= C_r x_r(t).
\end{align*}

To obtain such projections from the cross Gramian, various methods can be used.
The eigenvalue decomposition of the cross Gramian matrix,
\begin{align*}
 W_X \stackrel{\operatorname{EVD}}{=} T \Lambda T^{-1},
\end{align*}
given a symmetric system, yields a balancing projection $S := T$, $R :=T^{-1}$ \cite{aldhaheri91},
which can be truncated based on the absolute value of the magnitude of the eigenvalues $|\lambda_i| = |\Lambda_{ii}|$.
Alternatively, a singular value decomposition of the cross Gramian,
\begin{align*}
 W_X \stackrel{\operatorname{SVD}}{=} U \Sigma V,
\end{align*}
can be utilized.
Similarly, $S := U$ and $R := V$ can truncated based on the associated singular values $\sigma_i = \Sigma_{ii}$;
yet this projection is only approximately balancing \cite{sorensen02,rahrovani14} and the reduced order model's stability is not guaranteed to be preserved.

As a variant, only the left or right singular vectors can be used individually as a Galerkin projection,
\begin{align}\label{eq:dt}
 \begin{cases}
  S := U & R := U^\T \\
  S := V^\T & R := V.
 \end{cases}
\end{align}
This direct truncation \cite{himpe14a} is less accurate, but provides an orthogonal projection.

Lastly, we note that instead of truncating the decomposition derived projections based on the eigen- or singular values,
it is suggested in \cite{davidson86}, to use the quantities $d_i := |\tilde{b}_i \tilde{c}_i \lambda_i|$ and $\hat{d}_i := |\tilde{b}_i \tilde{c}_i \sigma_i|$ (compare \eqref{eq:h2}) for balanced and approximately balanced systems respectively,
which utilizes the columns of the (approximately) balanced input matrix $\tilde{b}_i$ and rows of the (approximately) balanced output matrix $\tilde{c}_i$.

\section{Inverse Sylvester Procedure}\label{sec:isp}
The proposed system generator is a special case of the inverse Lyapunov procedure \cite{smith03}.
This variant though generates exclusively state-space symmetric systems\footnote{For a state-space symmetric system $A=A^\T$ and $B = C^\T$ holds.},
which are found in applications such as RC circuits and have some interesting properties as shown \cite{liu98,reis14}.

We note that the cross Gramian, as an $N \times N$ dimensional linear operator $W_X : \R^N \to \R^N$, is an endomorphism.
This leads to the following relation between the system matrix $A$ and the cross Gramian matrix $W_X$, as stated in \cite{mironovskii15}:
\begin{mycorollary}
 Let $W_X$ be the cross Gramian to the system $(A,B,C)$.
 Then $A$ is \textbf{a} cross Gramian to the virtual system $(-W_X,B,C)$.
\end{mycorollary}
\begin{proof}
 This is a direct consequence of \href{lm:syl}{Lemma~\ref*{lm:syl}}.
\end{proof}
Hence, for a known cross Gramian $W_X$, input matrix $B$ and output matrix $C$,
the associated system matrix $A$ can be computed as the cross Gramian of the virtual system.
To ensure the (asymptotic) stability of the system, an observation from \cite[Theorem~2.1]{liu98} is utilized.
\begin{mylemma}
 For a state-space symmetric system the cross Gramian is symmetric and positive semi-definite.
\end{mylemma}
\begin{proof}
 Given a state-space symmetric system, the associated cross Gramian's Sylvester equation \eqref{eq:syl} becomes a Lyapunov equation:
 \begin{align*}
  A W_X + W_X A = BC \Leftrightarrow A W_X + W_X A^\T = BB^\T,
 \end{align*}
 of which a solution is symmetric and positive semi-definite.
\end{proof}
Thus, an (asymptotically) stable state-space symmetric system can be generated by providing an input matrix $B$,
which determines the output matrix $C = B^\T$ and a symmetric positive semi-definite cross Gramian $W_X$.
A procedure\footnote{See also \texttt{isp.m} in the associated source code archive.} to generate random asymptotically stable state-space symmetric systems, called \textbf{inverse Sylvester procedure},
is given by:

\begin{enumerate}
 \item Sample the cross Gramian's eigenvalues to define a positive definite cross Gramian in balanced form\footnote{A system Gramian in balanced form is a diagonal matrix.} from $\lambda_i = a (\frac{b}{a})^{\mathcal{U}_{[0,1]}}$ with $0<a<b$.
 \item Sample an input matrix $B$ from an iid multivariate standard normal distribution $\mathcal{N}_{0,1}^{N \times M}$ and set the output matrix to $C := B^\T$.
 \item Solve $-W_X A - A W_X = -BC \Leftrightarrow W_X A + A W_X = BC$ for (a negative semi-definite) system matrix $A$.
 \item Sample an orthogonal (un-)balancing transformation $U$ by a QR decomposition of a multivariate standard normally distributed matrix $U = \operatorname{qr}(\mathcal{N}_{0,1}^{N \times N})$.
 \item Unbalance the system by: $U^\T A U$, $U^\T B$, $C U$
\end{enumerate}

\section{Model Reduction Experiments}\label{sec:num}
In this section the Sylvester-equation-based cross Gramian is compared to the empirical cross Gramians in terms of state-space model reduction of a random system generated by the inverse Sylvester procedure.
A test system is generated as state-space symmetric SISO system, $M = Q = 1$, of order $N=1000$ by the inverse Sylvester procedure,
using $a=\frac{1}{10}$, $b=10$, excited by zero-mean, unit-variance Gaussian noise during each time-step and starting from a zero initial state.
Due to the use of empirical Gramians a time horizon of $T = 1$ and a fixed time-step width of $h = \frac{1}{100}$ is selected.
The cross Gramian variants are computed by solving a matrix equation\footnote{Since the system is state-space symmetric, practically a Lyapunov equation is solved.} \eqref{eq:syl}, by the empirical linear cross Gramian \eqref{eq:lwx} and the empirical cross Gramian \eqref{eq:nwx},
from which the reducing projections are obtained using the direct truncation approximate balancing\footnote{Using the SVD of the cross Gramian is equivalent to the eigendecomposition, due to the state-space symmetry.} \eqref{eq:dt} method.
The model reduction error, the error between the FOM (Full Order Model) and ROM (Reduced Order Model) output, is measured in the (time-domain) Lebesgue $L_1$, $L_2$ and $L_\infty$-norms,
\begin{align*}
 \|y-y_r\|_{L_1} &= \int_0^\infty \|y(t)-y_r(t)\|_1 \D t, \\
 \|y-y_r\|_{L_2} &= \sqrt{\int_0^\infty \|y(t)-y_r(t)\|_2^2 \D t}, \\
 \|y-y_r\|_{L_\infty} &= \operatornamewithlimits{ess\,sup}_{t \in [0,\infty)} \|y(t)-y_r(t)\|_\infty,
\end{align*}
as well as in the (frequency-domain) Hardy $H_\infty$-norm and approximately in the Hardy $H_2$-norm.
Due to the use of a state-space symmetric SISO system, twice the truncated tail of singular values not only bounds, but equals the $H_\infty$ error between the original and reduced transfer function $G$ and $G_r$ \cite[Theroem~4.1]{liu98}:
\begin{align*}
 \| G - G_r \|_{H_\infty} = 2 \sum_{i=n+1}^N \sigma_i.
\end{align*}
The $H_2$ error is approximated based on \cite[Remark~3.3]{sorensen02}:
\begin{align}\label{eq:h2}
 \|G - G_r\|_{H_2} \approx \sqrt{\tr(\widetilde{C}_2 W_{X,22} \widetilde{B}_2)},
\end{align}
with the balanced and truncated input and output matrices $\widetilde{B}_2 = B - U U^\T B$ and $\widetilde{C}_2 = C - C U U^\T$ and the truncated square lower right block of the balanced diagonal cross Gramian $W_{X,22}$.

\href{fig:l1}{Figure~\ref*{fig:l1}}, \href{fig:l2}{Figure~\ref*{fig:l2}} and \href{fig:l8}{Figure~\ref*{fig:l8}} show the relative $L_1$, $L_2$ and $L_\infty$ output errors for the classic, empirical linear and empirical cross Gramian using approximate-balancing-based projections over varying reduced state-space dimension up to order $\dim(x_r(t)) = 100$.

\begin{figure}[h!]
 \includegraphics[width=\textwidth]{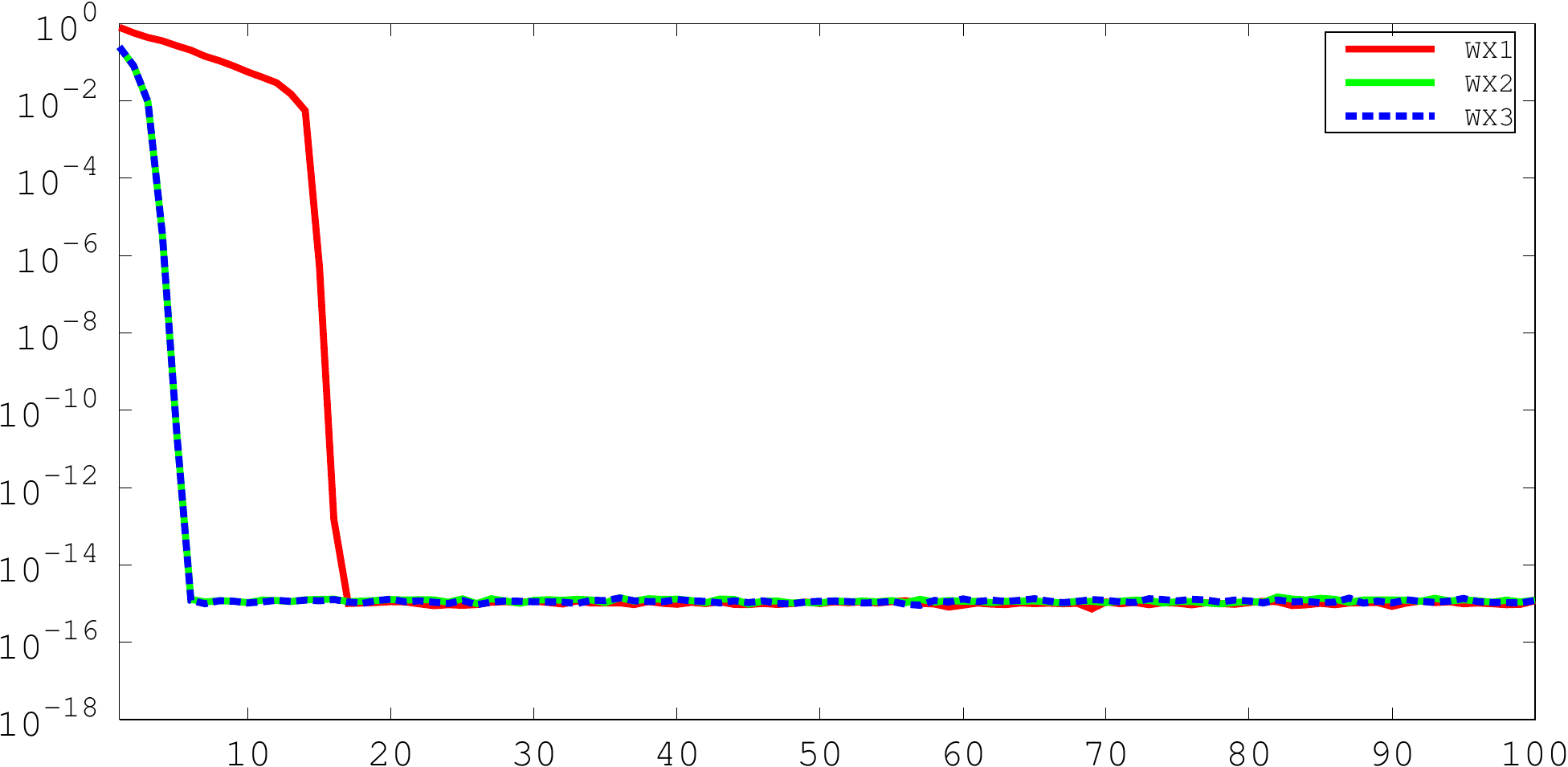}
 \caption{Relative $L_1$ output error between the FOM and ROMs for the matrix equation based cross Gramian $W_{X,1}$, the empirical linear cross Gramian $W_{X,2}$ and the empirical cross Gramian $W_{X,3}$}
 \label{fig:l1}
\end{figure}

\begin{figure}[h!]
 \includegraphics[width=\textwidth]{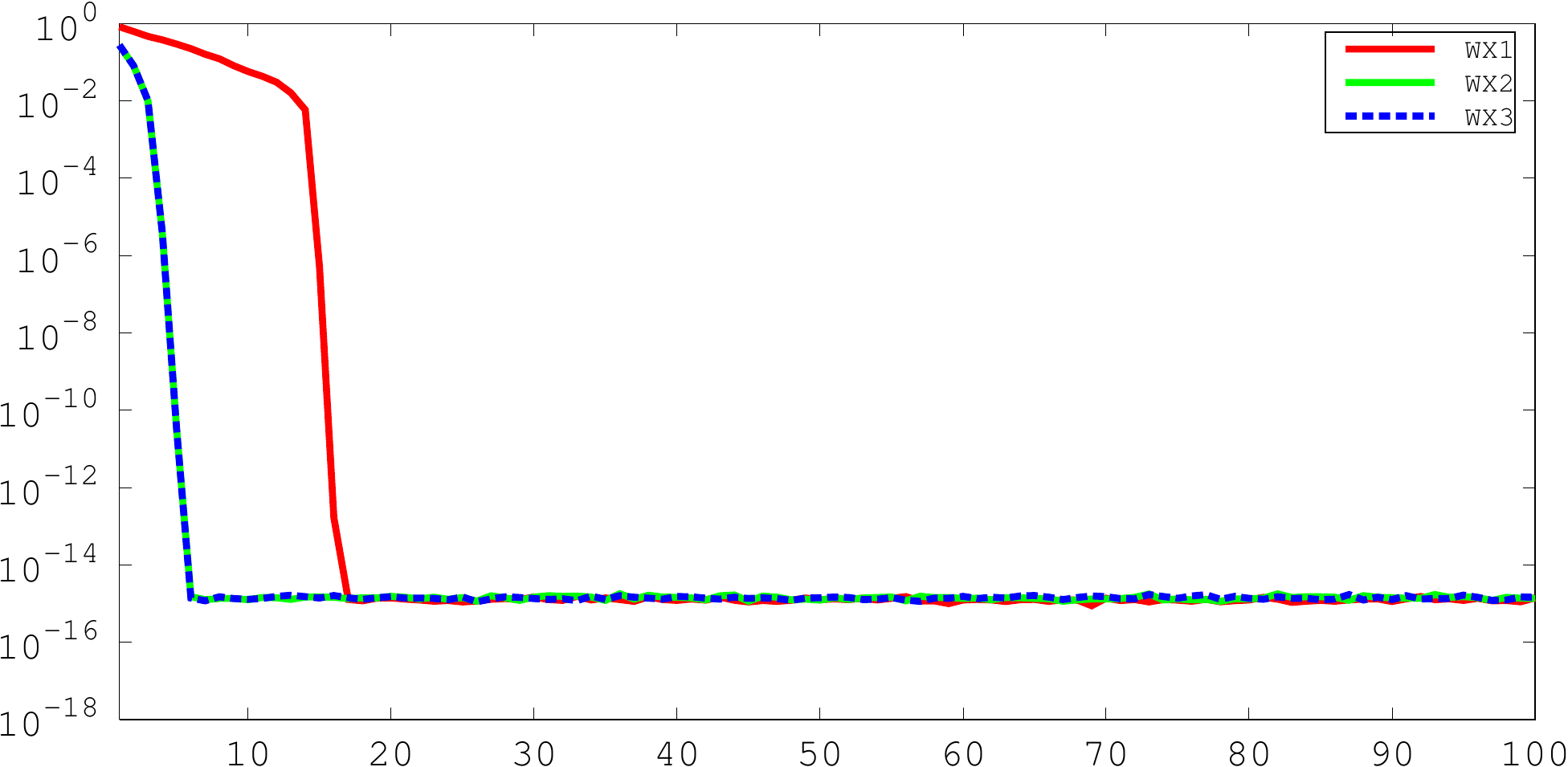}
 \caption{Relative $L_2$ output error between the FOM and ROMs for the matrix equation based cross Gramian $W_{X,1}$, the empirical linear cross Gramian $W_{X,2}$ and the empirical cross Gramian $W_{X,3}$}
 \label{fig:l2}
\end{figure}

\begin{figure}[h!]
 \includegraphics[width=\textwidth]{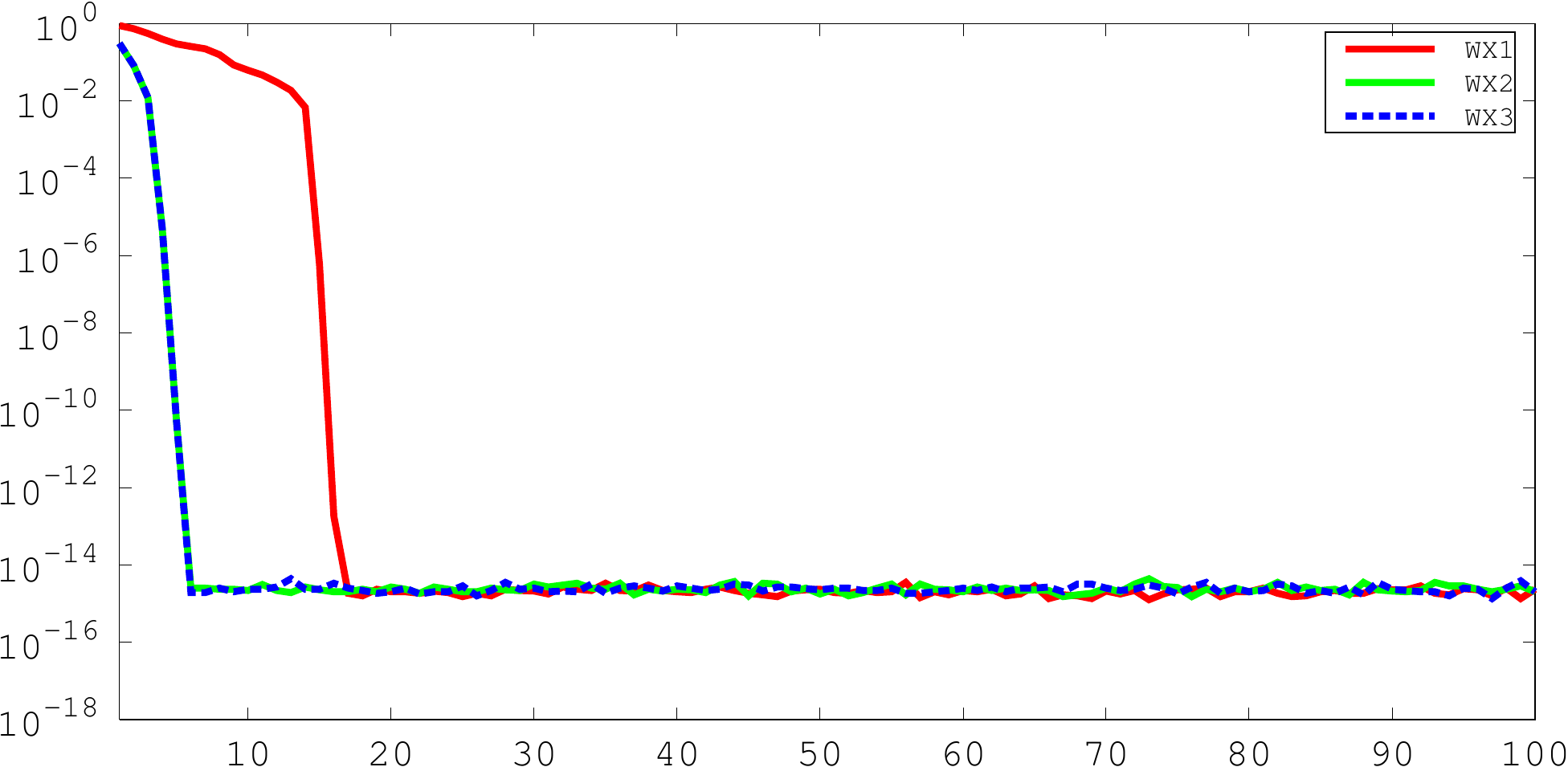}
 \caption{Relative $L_\infty$ output error between the FOM and ROMs for the matrix equation based cross Gramian $W_{X,1}$, the empirical linear cross Gramian $W_{X,2}$ and the empirical cross Gramian $W_{X,3}$}
 \label{fig:l8}
\end{figure}

For all tested cross Gramians, the Lebesgue error measures behave very similarly.
While the output errors for the empirical linear cross Gramian and the empirical cross Gramian decay at a reduced order of $n \ge 7$ to a level near the numerical precision with a similar rate,
the marix equation based cross Gramian reaches this level at $n \ge 18$.
Overall, the model reduces very well and due to the specific time frame for the reduction and comparison and the empirical Gramians yield better results.

In \href{fig:h2}{Figure~\ref*{fig:h2}} and \href{fig:h8}{Figure~\ref*{fig:h8}} the approximate $H_2$ error and the $H_\infty$ error are depicted for the three cross Gramian variants over varying reduced orders up to $\dim(x_r(t)) = 100$.

\begin{figure}[h!]
 \includegraphics[width=\textwidth]{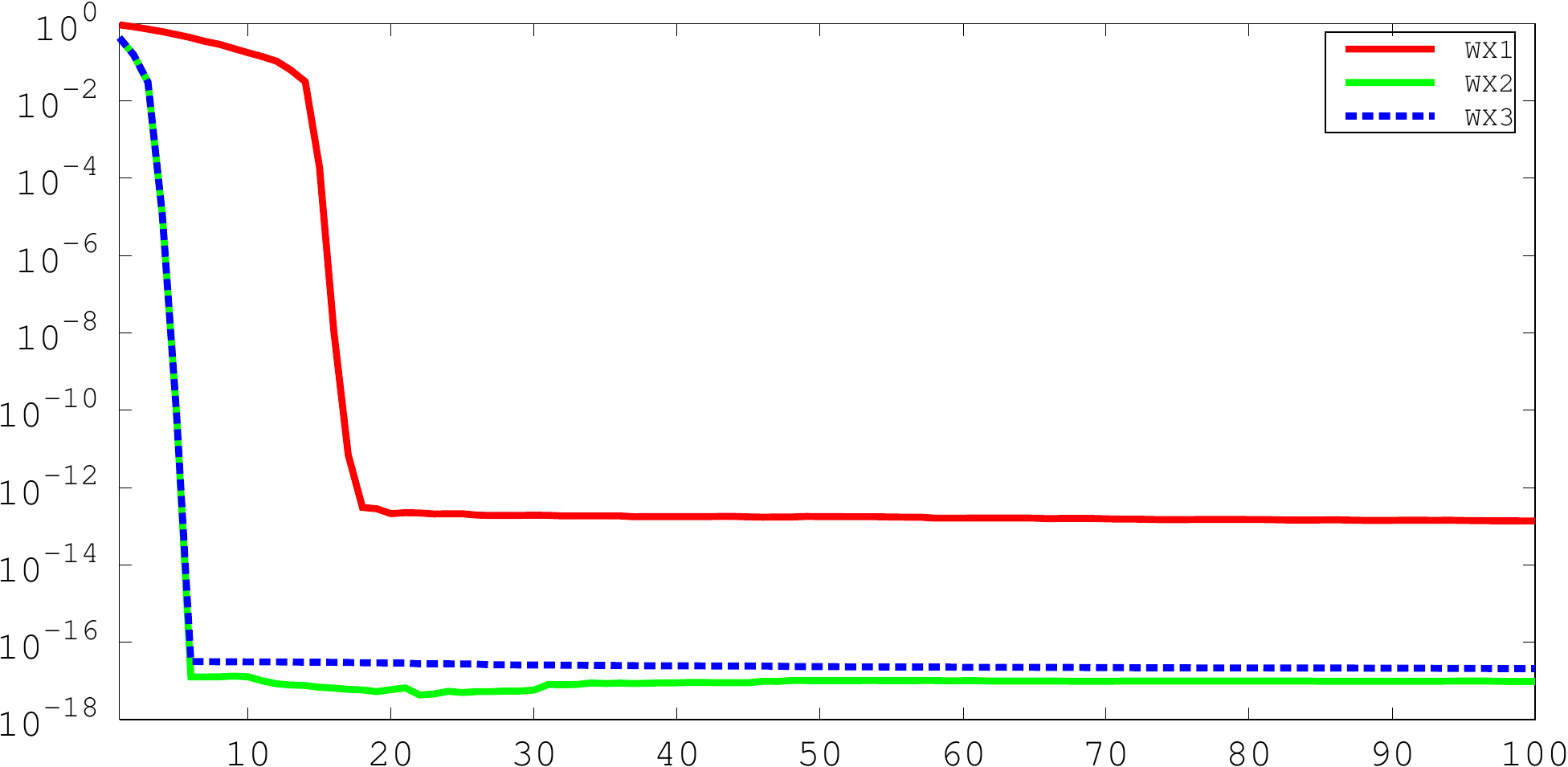}
 \caption{Approximate relative $H_2$ output error between the FOM and ROMs for the matrix equation based cross Gramian $W_{X,1}$, the empirical linear cross Gramian $W_{X,2}$ and the empirical cross Gramian $W_{X,3}$}
 \label{fig:h2}
\end{figure}

\begin{figure}[h!]
 \includegraphics[width=\textwidth]{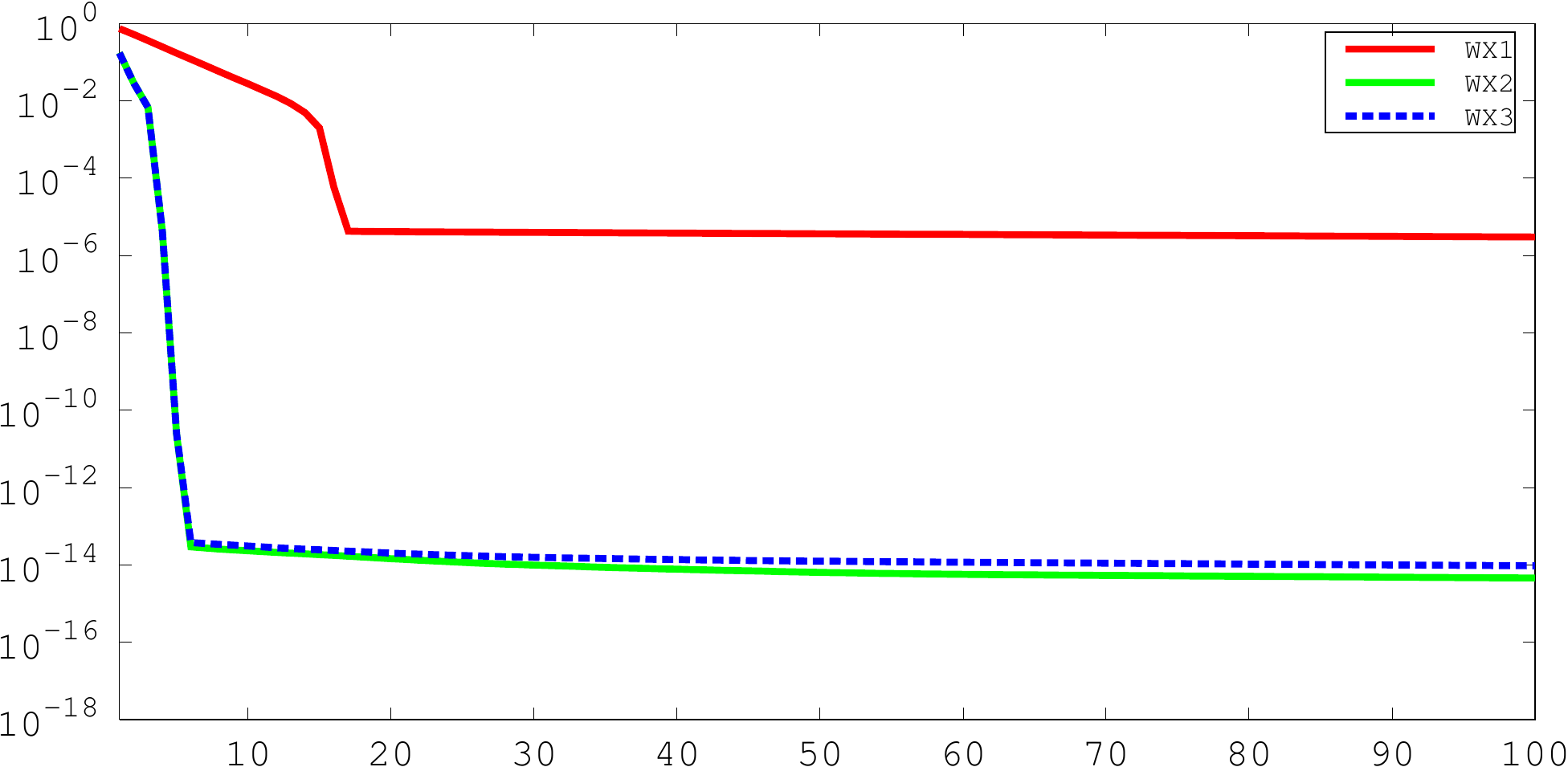}
 \caption{Relative $H_\infty$ output error between the FOM and ROMs for the matrix equation based cross Gramian $W_{X,1}$, the empirical linear cross Gramian $W_{X,2}$ and the empirical cross Gramian $W_{X,3}$}
 \label{fig:h8}
\end{figure}

For the frequency-domain errors the cross Gramian obtained as solution to a Sylvester (Lyapunov) equation does not attain the same accuracy as the empirical cross Gramians, which reach the numerical precision level for $n \ge 6$.
Also, as for the time-domain errors the sharp decay in the output error occurs at a higher reduced order $n \ge 19$ for the non-empirical cross Gramian,
but machine precision is not reached.

\section{Conclusion}
This work summarized the cross Gramian and its empirical variant and assesses methods for cross-Gramian-based model reduction mathematically and numerically.
The latter is conducted by a new cross-Gramian-based random state-space symmetric system generator.
Due to the strict definition of the operating region of the test system, the empirical cross Gramians produce superior reduced order models. 
This confirms the results of \cite{singh05}, that empirical Gramians can convey more information on the input-output behavior for a specific operating region than the classic matrix equation approach.

\section*{Code Availability}
The source code of the implementations used to compute the presented results can be obtained from: \\
\begin{center}
\url{http://www.runmycode.org/companion/view/1854} \\ and is authored by: Christian Himpe.
\end{center}

\section*{Acknowledgement}
This work was supported by the Deutsche Forschungsgemeinschaft: DFG EXC 1003 Cells in Motion - Cluster of Excellence, M\"unster, Germany 
and by the Center for Developing Mathematics in Interaction, DEMAIN, M\"unster, Germany.

\bibliographystyle{plain}
\bibliography{lit}

\begin{thebibliography}{10}

\bibitem{deabreu86}
J.A.~De Abreu-Garcia and F.W. Fairman.
\newblock {A Note on Cross Grammians for Orthogonally Symmetric Realizations}.
\newblock {\em IEEE Transactions on Automatic Control}, 31(9):866--868, 1986.

\bibitem{aldhaheri91}
R.W. Aldhaheri.
\newblock {Model order reduction via real Schur-form decomposition}.
\newblock {\em International Journal of Control}, 53(3):709--716, 1991.

\bibitem{baur16}
U.~Baur, P.~Benner, B.~Haasdonk, C.~Himpe, I.~Martini, and M.~Ohlberger.
\newblock {Comparison of Methods for Parametric Model Order Reduction of
  Instationary Problems}.
\newblock In P.~Benner, A.~Cohen, M.~Ohlberger, and K.~Willcox, editors, {\em
  Model Reduction and Approximation: Theory and Algorithms}, page To Appear.
  SIAM, 2016.

\bibitem{davidson86}
A.~Davidson.
\newblock {Balanced Systems and Model Reduction}.
\newblock {\em Electronics Letters}, 22(10):531--532, 1986.

\bibitem{fernando83a}
K.V. Fernando and H.~Nicholson.
\newblock {On the Structure of Balanced and Other Principal Representations of
  SISO Systems}.
\newblock {\em IEEE Transactions on Automatic Control}, 28(2):228--231, 1983.

\bibitem{fernando85}
K.V. Fernando and H.~Nicholson.
\newblock {On the Cross-Gramian for Symmetric MIMO Systems}.
\newblock {\em IEEE Transactions on Circuits and Systems}, 32(5):487--489,
  1985.

\bibitem{fu09}
J.-B. Fu, H.~Zhang, and Y.-X. Sun.
\newblock {Model reduction by minimizing information loss based on
  cross-Gramian matrix}.
\newblock {\em Journal of Zhejiang University (Engineering Science)},
  43(5):817--821, 2009.

\bibitem{glover84}
K.~Glover.
\newblock {All optimal Hankel-norm approximations of linear multivariable
  systems and their $\operatorname{L}^\infty$-error bounds}.
\newblock {\em International Journal of Control}, 39(6):1115--1193, 1984.

\bibitem{gray98}
W.S. Gray and J.M.A. Scherpen.
\newblock {Hankel Operators and Gramians for Nonlinear Systems}.
\newblock In {\em Proceedings of the 37th IEEE Conference on Decision and
  Control}, volume~2, pages 1416--1421, 1998.

\bibitem{himpe14a}
C.~Himpe and M.~Ohlberger.
\newblock {Cross-Gramian Based Combined State and Parameter Reduction for
  Large-Scale Control Systems}.
\newblock {\em Mathematical Problems in Engineering}, 2014:1--13, 2014.

\bibitem{himpe15a}
C.~Himpe and M.~Ohlberger.
\newblock {A Note on the Non-Symmetric Cross Gramian}.
\newblock {\em arXiv Preprint}, math.OC(1501.05519):1--6, 2015.

\bibitem{himpe15d}
C.~Himpe and M.~Ohlberger.
\newblock {The Versatile Cross Gramian}.
\newblock In {\em ScienceOpen Posters}, volume MoRePas 3, 2015.

\bibitem{hladnik88}
M.~Hladnik and M.~Omladi{\v{c}}.
\newblock {Spectrum of the Product of Operators}.
\newblock {\em Proceedings of the American Mathematical Society},
  102(2):300--302, 1988.

\bibitem{lall99}
S.~Lall, J.E. Marsden, and S.~Glavaski.
\newblock {Empirical Model Reduction of Controlled Nonlinear Systems}.
\newblock In {\em Proceedings of the 14th IFAC Congress}, volume~F, pages
  473--478, 1999.

\bibitem{lall02}
S.~Lall, J.E. Marsden, and S.~Glavaski.
\newblock {A subspace approach to balanced truncation for model reduction of
  nonlinear control systems}.
\newblock {\em International Journal of Robust and Nonlinear Control},
  12(6):519--535, 2002.

\bibitem{lancaster70}
P.~Lancaster.
\newblock {Explicit Solutions of Linear Matrix Equations}.
\newblock {\em SIAM Review}, 12(4):544--566, 1970.

\bibitem{laub83}
A.J. Laub, L.M. Silverman, and M.~Verma.
\newblock {A Note on Cross-Grammians for Symmetric Realizations}.
\newblock {\em Proceedings of the IEEE}, 71(7):904--905, 1983.

\bibitem{liu98}
W.Q. Liu, V.~Sreeram, and K.L. Teo.
\newblock {Model reduction for state-space symmetric systems}.
\newblock {\em Systems \& Control Letters}, 34(4):209--215, 1998.

\bibitem{mironovskii15}
L.A. Mironovskii and T.N. Solov'eva.
\newblock {Analysis of Multiplicity of Hankel Singular Values of Control
  Systems}.
\newblock {\em Automation and Remote Control}, 76(2):205--218, 2015.

\bibitem{moore81}
B.~Moore.
\newblock {Principal Component Analysis in Linear Systems: Controllability,
  Observability, and Model Reduction}.
\newblock {\em IEEE Transactions on Automatic Control}, 26(1):17--32, 1981.

\bibitem{moore79}
B.C. Moore.
\newblock {Principal Component Analysis in Nonlinear Systems: Preliminary
  Results}.
\newblock In {\em 18th IEEE Conference on Decision and Control including the
  Symposium on Adaptive Processes}, volume~2, pages 1057--1060, 1979.

\bibitem{reis14}
M.R. Opmeer and T.~Reis.
\newblock {The balanced truncation error bound in Schatten norms}.
\newblock {\em Hamburger Beitr{\"a}ge zur Angewandten Mathematik},
  2014(1):1--6, 2014.

\bibitem{rahrovani14}
S.~Rahrovani, M.K. Vakilzadeh, and T.~Abrahamsson.
\newblock {On Gramian-Based Techniques for Minimal Realization of Large-Scale
  Mechanical Systems}.
\newblock In {\em Topics in Modal Analysis}, volume~7, pages 797--805, 2014.

\bibitem{singh05}
A.K. Singh and J.~Hahn.
\newblock {On the Use of Empirical Gramians for Controllability and
  Observability Analysis}.
\newblock In {\em Proceedings of the American Control Conference}, volume 2005,
  pages 140--141, 2005.

\bibitem{smith03}
S.C. Smith and J.~Fisher.
\newblock {On generating random systems: a gramian approach}.
\newblock In {\em Proceedings of the American Control Conference}, volume~3,
  pages 2743--2748, 2003.

\bibitem{sorensen01}
D.C. Sorensen and A.C. Antoulas.
\newblock {Projection methods for balanced model reduction}.
\newblock Technical report, Rice University, 2001.

\bibitem{sorensen02}
D.C. Sorensen and A.C. Antoulas.
\newblock {The Sylvester equation and approximate balanced reduction}.
\newblock {\em Linear Algebra and its Applications}, 351--352:671--700, 2002.

\bibitem{streif06}
S.~Streif, R.~Findeisen, and E.~Bullinger.
\newblock {Relating Cross Gramians and Sensitivity Analysis in Systems
  Biology}.
\newblock {\em Theory of Networks and Systems}, 10.4:437--442, 2006.

\end{thebibliography}

\end{document}